\documentclass[12pt]{article}
\usepackage[a4paper, total={6in, 9in}]{geometry}
\usepackage{cite}
\usepackage{amsthm}
\usepackage{amsmath,amssymb,amsfonts}
\usepackage{bbm}

\usepackage{algorithmic}
\usepackage{graphicx}
\usepackage{algorithm,algorithmic}
\usepackage{hyperref}
\hypersetup{hidelinks=true}
\usepackage{textcomp}
\usepackage{graphicx} % Required for inserting images
\usepackage{mathtools}
\usepackage{amsfonts}

\usepackage{comment}
\usepackage{caption}
\usepackage{subcaption}

\usepackage{mathtools}

\DeclarePairedDelimiterX\braket[2]{\langle}{\rangle}{#1\,\delimsize\vert\,\mathopen{}#2}

\newtheorem{theorem}{Theorem}[section]
\newtheorem{remark}[theorem]{Remark}

\newtheorem{lemma}[theorem]{Lemma}

\usepackage{xcolor}
\usepackage{soul}

\begin{document}

\title{\LARGE \bf
An Approach to Control Design for Two-level Quantum Ensemble Systems*
}

\author{Ruikang Liang$^{1}$ and Gong Cheng$^{2}$% <-this % stops a space
\thanks{*RL's work has been partly supported by the ANR-DFG project ``CoRoMo'' ANR-22-CE92-0077-01. This project has received financial support from the
CNRS through the MITI interdisciplinary programs. GC's work is supported in part by the Fundamental Research Funds for the Central Universities and by NSFC Grant No.~12401588.}
\thanks{$^{1}$R.~Liang is with Sorbonne Universit\'e, Universit\'e 
Paris Cit\'e, CNRS, INRIA, Laboratoire Jacques-Louis Lions, LJLL, F-75005 Paris, France
    {\tt\small  ruikang.liang@sorbonne-universite.fr}}%
\thanks{$^{2}$G.~Cheng is with the School of Mathematical Sciences and the Key Laboratory of Intelligent Computing and Applications (Ministry of Education), Tongji University, Shanghai 200092, CHINA
    {\tt\small gongch@tongji.edu.cn}}%
}

\maketitle
\thispagestyle{empty}
\pagestyle{empty}

%%%%%%%%%%%%%%%%%%%%%%%%%%%%%%%%%%%%%%%%%%%%%%%%%%%%%%%%%%%%%%%%%%%%%%%%%%%%%%%%
\begin{abstract}

Quantum ensemble systems arise in a variety of applications, including NMR spectroscopy and robust quantum control. While their theoretical properties have been extensively studied, relatively little attention has been given to the explicit construction of control inputs. In this paper, we address this gap by presenting a fully implementable control strategy for a one-parameter family of driftless two-level quantum systems. The proposed method is supported by rigorous analysis that guarantees accurate approximation of target distributions on $\mathbf{SU(2)}$. Convergence properties are established analytically, and numerical simulations are provided to demonstrate the effectiveness of the approach.

\end{abstract}

\section{INTRODUCTION}
\label{sec:intro}

In this paper, we address the ensemble control problem for a continuum of driftless two-level quantum systems on $\mathrm{SU}(2)$, steered by two control inputs. The system dynamics is given by
\begin{equation}
\label{system:intro}
    i\dot{X}(t)=\omega
    \begin{pmatrix}
        0 & u(t)-iv(t) \\
        u(t) + iv(t) & 0
    \end{pmatrix}X(t),
\end{equation}
where $\omega$ is a parameter ranging over an interval $[a,b]\subset\mathbb{R}^{*}_{+}$. The objective is to design a single (uniform) control law that steers the entire family of systems from a common initial state to a collection of target states on $\mathrm{SU(2)}$, where the target depends continuously on the parameter $\omega$.

Controlling large ensembles of two-level quantum systems with uniform control inputs is a central challenge in quantum control, due to inherent parameter variations across the ensemble. Prior research has established that such inhomogeneous two-level systems are, in principle, controllable, often through Lie-algebraic criteria and functional analytic techniques.
%Previous studies have primarily focused on establishing ensemble controllability for two-level quantum systems.
In particular, controllability of system~\eqref{system:intro} was first demonstrated in \cite{li2006bloch}. Related models involving a family of systems with dispersion in frequency and driven by a bounded complex-valued control $u(t)$ were analyzed in \cite{BeauchardCoronRouchon, li2009ensemble}:
\[
    i\dot{X}(t)=\begin{pmatrix}\omega & u(t) \\
    \bar{u}(t) & -\omega\end{pmatrix}X(t).
\]
Extensions to the case of real-valued bounded controls were later presented in \cite{ROBIN2022414}.

Despite these advances in theoretical controllability, the explicit construction of control inputs for such systems remains largely underexplored. Initial efforts have focused on population transfer via adiabatic motion. For instance, in \cite{liang:10886609, liang:hal-04819595, ROBIN2022414}, it was shown that uniform population transfer between eigenstates can be achieved for a continuum of $n$-level systems using a scalar control that combines rotating wave and adiabatic approximations. Further developments can be found in \cite{AugierSIAM, AugierMCRF, LeghtasSarletteRouchon}. However, a general and practical design method remains largely elusive.

In contrast to these approaches, which are often theoretical or restricted to specific transfer tasks, the present work provides a constructive procedure to synthesize control fields that steer the entire ensemble toward a desired target distribution - beyond population transfer between eigenstates. By recasting the system in a suitable time-dependent frame, we show that rotational terms with frequencies depending on $\omega$ naturally arise in the transformed dynamics. Leveraging this structure, we construct control inputs based on Fourier transform techniques, enabling the approximation of arbitrary smooth target distributions as functions of $\omega$. The effectiveness of the proposed method is demonstrated through numerical simulations, which confirm that the analytically derived controls accurately realize the intended quantum-state transformations across the ensemble.

The remainder of the paper is organized as follows. Section~\ref{sect:result} presents the explicit construction of the control law and states the main theoretical result. Section~\ref{sect:numerical} provides numerical experiments that illustrate the effectiveness of the proposed control strategy in two representative scenarios.

\section{MAIN RESULT}
\label{sect:result}

As a first step toward our main result, we reformulate system~\eqref{system:intro} in terms of the Pauli matrices as follows: 
\begin{equation}
\label{system:main}
    i\dot{X}(\omega,t) = \omega\bigl( u(t)\sigma_{x} + v(t)\sigma_{y}\bigr)X(\omega,t),
\end{equation}
with initial state $ X(\omega,0) = \mathbb{I}$, where $u(\cdot)$ and $v(\cdot)$ and $\omega$ are as defined in \eqref{system:intro}. The matrices $\sigma_{x},\sigma_y$, and $\sigma_{z}$ denote the standard Pauli matrices:
\[
    \sigma_{x}=\begin{pmatrix}
        0 & 1\\
        1 & 0
    \end{pmatrix},\ %
    \sigma_{y}=\begin{pmatrix}
        0 & -i\\
        i & 0
    \end{pmatrix},\ %
    \sigma_{z}=\begin{pmatrix}
        1 & 0\\
        0 & -1
    \end{pmatrix}.
\]

\subsection{Fourier transform}
To facilitate the control design of the driftless system presented in (1), this section introduces a set of auxiliary functions and revisits key properties of their Fourier transforms. These properties are fundamental to the control strategy developed in the subsequent sections. For a compact subset $K$ of $\mathbb{R}$, we introduce the following notation
\begin{equation}
\label{eq:notation-compactly-supported}
    \mathcal{C}_{c}^{\infty}(K):=\bigl\{f\in\mathcal{C}_{c}^{\infty}(\mathbb{R})\mid\operatorname{supp} f\subset K\bigr\},
\end{equation}
which denotes the set of smooth functions supported on the compact set $K$. Take $f\in\mathcal{C}_{c}^{\infty}([v_0,v_1])$ with $0<v_0<v_1$, and let $g$ be an even function associated with $f$ defined by
\begin{equation}
\label{eq:def_g}
    \omega g(2\omega) := f(\omega), \quad \forall\,\omega\geq 0,
\end{equation}
which is a smooth function compactly supported on $[-2v_1,-2v_0]\cup[2v_0,2v_1]$. Now consider the unitary Fourier transform of $g$:
\begin{equation}
\begin{aligned}
\label{eq:def_hat_g}
    \hat{g}(t)&=\frac{1}{\sqrt{2\pi}}\int_{-\infty}^{\infty}g(\omega)e^{-it\omega}\text{d}\omega\\
    &=\sqrt{\frac{2}{\pi}}\int_{2v_0}^{2v_1}g(\omega)\cos(t\omega)\text{d}\omega,
\end{aligned}
\end{equation}
with $t\in\mathbb{R}$.
Since $g(\cdot)$ is even, its Fourier transform $\hat{g}$ has only real part and is also even. By the inverse transform we have that for all $\omega\in\mathbb{R}$,
\begin{equation}
\label{eq:fourier_inverse}
    \begin{aligned}
        \frac{1}{\sqrt{2\pi}}\int_{-\infty}^{\infty} \hat{g}(\tau)\cos(\omega \tau)\text{d}\tau &= g(\omega), \\
        \frac{1}{\sqrt{2\pi}}\int_{-\infty}^{\infty} \hat{g}(\tau)\sin(\omega \tau)\text{d}\tau &= 0.
    \end{aligned}
\end{equation}

\begin{comment}
{Now let us take $\bar{u}\equiv 1$ and recast the system \eqref{system:main} in the interaction frame $X_{I}(\omega,t)=\exp(it\omega \sigma_{x})X(\omega,t)$. The dynamics of $X_{I}(t)$ will be characterized by
\[
    iX'_{I}(t)=\omega\Big(\sin(\omega t)v(t)\sigma_{z}+\cos(\omega t)v(t)\sigma_{y}\Big)X_{I}(\omega,t),\qquad X_{I}(\omega,0)=\mathbb{I}.
\]}
\end{comment}

To aid the computations involved in the subsequent control design, we present the following properties of $g$ and its Fourier transform.

\begin{lemma}
\label{lemma:schwartz}
    For all $n\in\mathbb{N}^{*}$, there exists $C_n>0$ such that for all $t\in\mathbb{R}^{*}$, we have $|\hat{g}(t)|\leq C_{n}|t|^{-n}$. Thus, $\hat{g}\in L^1(\mathbb{R})$.
\end{lemma}

\begin{proof}
This is obvious, since the unitary Fourier transformation is an isomorphism on the Schwartz space
\begin{equation}
    \label{eq:schwartz}
    \mathcal{S}=\{f\in C^{\infty}\mid \sup_{x\in\mathbb{R}}(1+|x|)^N|\partial^{\alpha} f(x)| < \infty, \forall\, N, \alpha\},
\end{equation}
and since $g\in C_{c}^{\infty}(\mathbb{R})\subset \mathcal{S}$, we must have $\hat{g}\in\mathcal{S}\subset L^1(\mathbb{R})$.
% Let us take $t\neq 0$. By partial integration, for the compactly supported smooth function defined in equation~\eqref{eq:def_g}, we  {\color{red}(tempered distribution)}
%     \[
%     \begin{aligned}
%        \hat{g}(t)= \int_{-\infty}^{\infty}\frac{1}{\sqrt{2\pi}}g(\omega)e^{-i\omega t}\text{d}\omega&=\left[\frac{i}{\sqrt{2\pi} t}g(\omega)e^{-i\omega t}\right]_{\omega=-\infty}^{\omega=\infty}-\int_{-\infty}^{\infty}\frac{i}{\sqrt{2\pi} t}g'(\omega)e^{-i\omega t}\text{d}\omega\\
%         &=-\int_{-\infty}^{\infty}\frac{i}{\sqrt{2\pi} t}g'(\omega)e^{-i\omega t}\text{d}\omega
%     \end{aligned}
%     \]
% Since for all $n\in\mathbb{N}^{*}$, the $n$-th order derivative of $g$ is a smooth function with the same compact support $[-2v_1,-2v_0]\cup[2v_0,2v_1]$, then by repeating $n$ times the partial integration above for compactly supported smooth function, we can obtain that 
% \[
%     \hat{g}(t)=\int_{-\infty}^{\infty} \frac{(-i)^{n}}{\sqrt{2\pi}t^{n}}g^{(n)}(\omega)e^{-i\omega t}\text{d}\omega
% \]
% Then for all $n\in\mathbb{{N}}^{*}$ and $t\neq 0$,
% \[
%     |\hat{g}(t)|\leq\frac{1}{\sqrt{2\pi} |t|^{n}}\int_{-\infty}^{\infty}\left|g^{(n)}(\omega)\right|\text{d}\omega.
% \]
% Since $g^{(n)}$ is compactly supported, we can conclude the proof by taking $C_{n}=1/\sqrt{2\pi}\int_{-\infty}^{\infty}|g^{(n)}(\omega)|\text{d}\omega$, which is always finite.
\end{proof}

\begin{lemma}
\label{lemma:integral_2}
    For every $n\in\mathbb{N}^{*}$, there exists $C_{n}>0$ such that for all $\epsilon_1>0$ and $\omega\in[2v_0,2v_1]$, we have
    \begin{equation}
    \label{eq:integral_2}
        \biggl| g(\omega)-\int_{-\frac{1}{\epsilon_1}}^{\frac{1}{\epsilon_1}}\frac{1}{\sqrt{2\pi}}\hat{g}(\tau)\cos(w\tau)\text{d}\tau \biggr| < C_{n} \epsilon_1^{-n}.
    \end{equation}
\end{lemma}

\begin{proof}
Take $n\in\mathbb{N}^{*}$. By Lemma~\ref{lemma:schwartz}, there exists $C'_{n+1}$ such that for all $t\neq 0$, $|\hat{g}(t)|<C'_{n+1}|t|^{-(n+1)}$. Then we can obtain that, for $\epsilon_1>0$,
\[
    \begin{aligned}
    \int_{\frac{1}{\epsilon_1}}^{\infty}|\hat{g}(\tau)|\text{d}\tau&\leq\int_{\frac{1}{\epsilon_1}}^{\infty}\frac{C'_{n+1}}{\tau^{n+1}}\text{d}\tau=\frac{C'_{n+1}}{n}\epsilon_1^{n}
    \end{aligned}
\]
By equation~\eqref{eq:fourier_inverse}, and since $\hat{g}(t)$ is even, for $\epsilon_1>0$,
\[
\begin{aligned}
     \biggl| g(\omega) - \int_{-\frac{1}{\epsilon_1}}^{\frac{1}{\epsilon_1}}\frac{1}{\sqrt{2\pi}}\hat{g}(\tau)\cos(\omega \tau)\text{d}\tau \biggr|
        &= 2\biggl|\int_{\frac{1}{\epsilon_1}}^{\infty}\frac{1}{\sqrt{2\pi}}\hat{g}(\tau)\cos(\omega \tau)\text{d}\tau \biggr| \\
        &< 2\int_{\frac{1}{\epsilon_1}}^{\infty}\frac{1}{\sqrt{2\pi}}|\hat{g}(\tau)|\text{d}\tau\leq \frac{2C'_{n+1}}{n\sqrt{2\pi}}\epsilon_1^{n}.
\end{aligned}
\]
We can thus conclude the proof by taking $C_{n}=2C'_{n+1}/(n\sqrt{2\pi})$.
\end{proof}
\begin{lemma}
\label{lemma:integral_3}
    For all $\epsilon_1>0$ and $\omega\in[2v_0,2v_1]$ we have 
    \[
        \int_{-\frac{1}{\epsilon_1}}^{\frac{1}{\epsilon_1}}\frac{1}{\sqrt{2\pi}}\hat{g}(\tau)\sin(\omega\tau)\text{d}\tau=0.
    \]
\end{lemma}
\begin{proof}
    The proof follows immediately by observing that $\hat{g}(t)$ is even, $\sin(\omega t)$ is odd, and the interval $[-1/\epsilon_1,1/\epsilon_1]$ is symmetric about the origin.
\end{proof}

\subsection{Control design}

In this section, we address the problem of control design of the driftless system described in \eqref{system:main}. To that end, consider a smooth function $f$, supported on a finite interval  $[v_0,v_1]\subset \mathbb{R}$, as introduced in the previous section. We then define an even function $\hat{g}$, associated with $f$, as in \eqref{eq:def_hat_g}, and fix constants $\epsilon_1>0$ and $N>0$. With these definitions in place, we then consider the following control inputs $u(\cdot)$ and $v(\cdot)$, defined on the interval $[0,(4N+2)/\epsilon_1]$:
\smallskip
\begin{itemize}
\setlength\itemsep{1em}
    \item $u(t)=-1$ and $v(t)=0$, if $t\in[0,1/\epsilon_1]$;
    
    \item $u(t)=1$ and $v(t)=\tfrac{\hat{g}(t-(4m+2)/\epsilon_1)}{2N\sqrt{2\pi}}$, if $t\in[\tfrac{4m+1}{\epsilon_1},\tfrac{4m+3}{\epsilon_1}]$ with $m\in\{0,\dots,N-1\}$;
    
    \item $u(t)=-1$ and $v(t)=\tfrac{\hat{g}(t-(4m+4)/\epsilon_1)}{2N\sqrt{2\pi}}$ if $t\in[\tfrac{4m+3}{\epsilon_1},\tfrac{4m+5}{\epsilon_1}]$ with $m\in\{0,\dots,N-1\}$;
    
    \item $u(t)=1$ and $v(t)=0$ if $t\in[\tfrac{4N+1}{\epsilon_1},\tfrac{4N+2}{\epsilon_1}]$.
\end{itemize}

With the control inputs defined above, we now turn to analyzing the resulting system behavior. The goal is to show that these controls steer the system from the identity to a specific ensemble of unitary matrices.

\begin{theorem}[Main result]
\label{theorem:main}
    Let $f\in C_{c}^{\infty}([v_0,v_1])$ with $0<v_{0}<v_1$, and fix $\epsilon_1>0$, $N\in\mathbb{N}^{*}$, and $n\in\mathbb{N}^{*}$. Consider the control inputs $u(\cdot)$ and $v(\cdot)$ defined as above. Then, the corresponding solution of system~\eqref{system:main} satisfies
    \[
        X\Bigl(\omega,\frac{4N+2}{\epsilon_1}\Bigr) = \exp\bigl(-if(\omega)\sigma_{y}\bigr) + \mathcal{O}(\epsilon_1^{n}+N^{-1}),
    \]
    where $\mathcal{O}(\cdot)$ is uniform for all $\omega\in[v_0,v_1]$.
\end{theorem}

The above theorem establishes that system~\eqref{system:main} can be steered from identity to an ensemble of exponential of the Pauli matrix $\sigma_y$. In the following remarks, we explain why this result is, in fact, sufficient to achieve full controllability over the ensemble. Specifically, we show how the stated result implies that, by leveraging symmetry and utilizing the Euler angle decomposition, the system can be steered from the identity to any desired distribution on $\mathrm{SU}(2)$.

\begin{remark}
    By symmetry, interchanging the control inputs $u(\cdot)$ and $v(\cdot)$ in the above construction yields $\exp(-if(\omega)\sigma_{x})$ instead of $\exp(-if(\omega)\sigma_{y})$. It will be further illustrated through numerical examples presented later in the paper.
\end{remark}

\begin{remark}[Euler angle decomposition]
According to the Euler angle decomposition (see, for example, \cite{Nielsen_Chuang_2010}), any element $U$ in $\mathrm{SU}(2)$ can be expressed in the form
\[
    U=\exp(-i\alpha\sigma_{x})\exp(-i\beta\sigma_y)\exp(-i\gamma\sigma_{x}),
\]
for some angles $\alpha, \beta, \gamma \in [0, 2\pi)$. Therefore, it suffices to approximate arbitrary smooth ensembles on $\mathrm{SU}(2)$ of the form $\exp(-if(\omega)\sigma_{x})$ and $\exp(-if(\omega)\sigma_{y})$ using our control strategy.
\end{remark}

Before presenting the proof of Theorem~\ref{theorem:main}, we introduce the following auxiliary system. For fixed parameters $\epsilon_1, \epsilon_2 > 0$, consider the system defined on the closed time interval $[-\tfrac{1}{\epsilon_1}, \tfrac{1}{\epsilon_1}]$:
\begin{equation}
\label{eq:auxiliary-system}
    i\dot{\hat{X}}(\omega,t)=\hat{H}_{\epsilon_1,\epsilon_2}(\omega,t)\hat{X}(\omega,t), \ \hat{X}\Bigl(\omega,-\frac{1}{\epsilon_1}\Bigr)=\mathbb{I},
\end{equation}
where the controlled Hamiltonian $\hat{H}_{\epsilon_1,\epsilon_2}$ is given by
\[
    \hat{H}_{\epsilon_1,\epsilon_2}(\omega,t) := \omega v_{\epsilon_1,\epsilon_2}(t)\bigl(\sin(2\omega t)\sigma_z+\cos(2\omega t)\sigma_{y}\bigr)
\]
and the control input $v_{\epsilon_1, \epsilon_2}$ is defined by
\begin{equation}
\label{eq:def_control}
    v_{\epsilon_1,\epsilon_2}:\left[-\frac{1}{\epsilon_1},\frac{1}{\epsilon_1}\right]\ni t\mapsto \frac{\epsilon_2}{\sqrt{2\pi}}\hat{g}(t),
\end{equation}
where $\hat{g}$ is the even function defined in \eqref{eq:def_hat_g} corresponding to $f\in\mathcal{C}_{c}^{\infty}([v_0,v_1])$. We then have the following lemma on the approximation of the ensemble in \eqref{eq:auxiliary-system}.

\begin{lemma}
\label{lemma:one_step}
    For every $n\in\mathbb{N}^{*}$, there exists a constant $C_{n}$ such that for all $\omega\in[-v_{1},-v_{0}]\cup[v_0,v_1]$, the solution $\hat{X}(\omega,\cdot)$ to \eqref{eq:auxiliary-system} associated with the control $v_{\epsilon_1,\epsilon_2}(\cdot)$ satisfies that
    \[
        \|\hat{X}(\omega,\epsilon_1^{-1}) - \exp(-i\epsilon_2\omega g(2\omega)\sigma_{y})\| < C_{n}(\epsilon_1^{n}\epsilon_2+\epsilon_2^2).
    \]
\end{lemma}

\begin{proof}
Fix $\epsilon_1,\epsilon_2>0$. Let us introduce the following change of variables
\begin{equation}
\label{eq:change_of_variables}
    \hat{X}(\omega,t)=\exp(-iI(\omega,t))\tilde{X}(\omega,t),
\end{equation}
where $I(\omega,t)$ is defined as
\begin{equation}
\label{eq:def_I}
    \bigl([-v_1,-v_0]\cup[v_0,v_1]\bigr) \times [-\frac{1}{\epsilon_1},\frac{1}{\epsilon_1}]\ni(\omega,t)
    \mapsto I(\omega,t)=\int_{-1/\epsilon_1}^{t}\hat{H}_{\epsilon_1,\epsilon_2}(\omega,\tau)\text{d}\tau.
\end{equation}
For $\omega\in[-v_1,-v_0]\cup[v_0,v_1]$, by \cite[Proposition~5]{liang:hal-04819595}, and since $\dot{I}(\omega,t)=\hat{H}_{\epsilon_1,\epsilon_2}(\omega,t)$, $\tilde{X}(\omega,\cdot)$ is the solution of 
\[
    i\dot{\tilde{X}}(\omega,t)=\tilde{H}_{\epsilon_1,\epsilon_2}(\omega,t)\tilde{X}(\omega,t),\ %
    \tilde{X}\Bigl(\omega,-\frac{1}{\epsilon_1}\Bigr) = \mathbb{I},
\]
where $\hat{H}_{\epsilon_1,\epsilon_2}(t)$ is given by
\[
    \tilde{H}_{\epsilon_1,\epsilon_2}(\omega,t)
    =\sum_{k=1}^{\infty}\frac{(i)^{k}k}{(k+1)!}\operatorname{ad}^{k}_{I(\omega,t)}\big(\hat{H}_{\epsilon_1,\epsilon_2}(\omega,t)\big).
\]
% \begin{comment}
% \[
%     \begin{aligned}
%          \tilde{H}_{\epsilon_1,\epsilon_2}(\omega,t)&=\sum_{k=0}^{\infty}\frac{(-1)^{k}}{k!}\text{ad}^{k}_{-iI(\omega,t)}\left(\hat{H}_{\epsilon_1,\epsilon_2}(\omega,t)-\frac{1}{k+1}I'(\omega,t)\right)\\
% &=\sum_{k=1}^{\infty}\frac{(i)^{k}k}{(k+1)!}\text{ad}^{k}_{I(\omega,t)}\big(\hat{H}_{\epsilon_1,\epsilon_2}(\omega,t)\big).
%     \end{aligned}
% \]
% \end{comment}
Notice that 
\[ 
    \|\tilde{H}_{\epsilon_1,\epsilon_2}\omega,t)\| \leq \sum_{k=1}^{\infty}\frac{2^{k}k}{(k+1)!} \|I(\omega,t)\|^{k} \|\hat{H}_{\epsilon_1,\epsilon_2}(\omega,t)\|.
\]
By \eqref{eq:def_control}, \eqref{eq:def_I} and Lemma~\ref{lemma:schwartz}, there exists $C>0$ such that, for all $\omega\in[-v_1,-v_0]\cup[v_0,v_1]$, we have
\[
\begin{aligned}
    \int_{-\frac{1}{\epsilon_1}}^{\frac{1}{\epsilon_1}} \|\hat{H}_{\epsilon_1,\epsilon_2}(\omega,\tau)\|\text{d}\tau &\leq \int_{-\frac{1}{\epsilon_1}}^{\frac{1}{\epsilon_1}} \epsilon_2 \frac{|\omega|}{\sqrt{2\pi}}|\hat{g}(\tau)|\text{d}\tau< C|\omega|\epsilon_2,
\end{aligned}
\]
and for all $t\in[-\tfrac{1}{\epsilon_1}, \tfrac{1}{\epsilon_1}]$,
\[
    \|I(\omega,t)\| \leq \int_{-\frac{1}{\epsilon_1}}^{t}\|\hat{H}_{\epsilon_1,\epsilon_2}(\omega,\tau)\|\text{d}\tau < C|\omega|\epsilon_2.
\]
Therefore, there exists $M>0$ such that for all $t\in[-\tfrac{1}{\epsilon_1}, \tfrac{1}{\epsilon_1}]$ and $\omega\in[-v_1,-v_0]\cup[v_0,v_1]$,
\begin{equation*}
\begin{aligned}
    \int_{-\frac{1}{\epsilon_1}}^{t}\left\|\tilde{H}_{\epsilon_1,\epsilon_2}(\omega,\tau)\right\|\text{d}\tau
    &\leq\sum_{k=1}^{\infty}\frac{2^{k}k}{(k+1)!}C^{k}|\omega|^{k}\epsilon_2^{k}\int_{-\frac{1}{\epsilon_1}}^{t}\|\hat{H}_{\epsilon_1,\epsilon_2}(\tau)\|\text{d}\tau \\
    &\leq\sum_{k=1}^{\infty}\frac{2^k k}{(k+1)!}C^{k+1}|\omega|^{k+1}\epsilon_2^{k+1}<M\epsilon_2^2.
\end{aligned}
\end{equation*}

Here $M$ is uniform for all $\epsilon_1,\epsilon_2>0$. So by Gr{\"o}nwall's inequality, there exists $M'>0$ such that for all $\epsilon_1,\epsilon_2>0$ and $\omega\in[-v_1,-v_0]\cup[v_0,v_1]$, the solution $\tilde{X}_{\epsilon_1,\epsilon_2}(\cdot)$ satisfies
\[
    \|\tilde{X}_{\epsilon_1,\epsilon_2}(\omega,\epsilon_1^{-1}) - \mathbb{I}\|<M'\epsilon_2^2.
\]
Apply the change of variables in \eqref{eq:change_of_variables} again, we deduce that
\begin{equation}
\label{eq:estimation_1}
    \|\hat{X}_{\epsilon_1,\epsilon_2}(\omega,\epsilon_1^{-1}) - \exp\bigl(-iI(\omega,\frac{1}{\epsilon_1})\bigr)\| \leq M'\epsilon_{2}^{2}.
\end{equation}
Recall that, for all $\omega\in[-v_1,-v_0]\cup[v_0,v_1]$,
\begin{equation*}
I(\omega, \epsilon_1^{-1}) = \int_{-\frac{1}{\epsilon_1}}^{\frac{1}{\epsilon_1}}\hat{H}_{\epsilon_1,\epsilon_2}(\tau)\text{d}\tau
=\int_{-\frac{1}{\epsilon_1}}^{\frac{1}{\epsilon_1}}\frac{\epsilon_2\omega}{\sqrt{2\pi}}\hat{g}(\tau)\bigl(\sin(2\omega\tau)\sigma_{z}+\cos(2\omega\tau)\sigma_{y}\bigr)\text{d}\tau.
\end{equation*}
Hence by Lemma~\ref{lemma:integral_2} and Lemma~\ref{lemma:integral_3}, for all $n\in\mathbb{N}^{*}$, there exists $C'_n>0$ such that, for all $\epsilon_1,\epsilon_2>0$ and $\omega\in[v_0,v_1]$,
\[
    \|\int_{-\frac{1}{\epsilon_1}}^{\frac{1}{\epsilon_1}}\hat{H}_{\epsilon_1,\epsilon_2}(\omega,\tau)\text{d}\tau-\epsilon_2\omega g(2\omega)\sigma_{y}\| \leq |\omega|C'_{n}\epsilon_1^{n}.
\]
It follows that there exists $C_{n}''>0$ satisfying for all $ \epsilon_1,\epsilon_2>0$ and for all $\omega\in[-v_1,-v_0]\cup[v_0,v_1]$, we have
\begin{equation}
\label{eq:estimation_2}
    \|\exp\bigl(-iI(\omega,\epsilon_1^{-1})\bigr) - \exp(-i\epsilon_2\omega g(2\omega)\sigma_{y})\| \\
    \leq C''_{n}\epsilon_1^{n}.
\end{equation}
Therefore, we have
\[ 
\begin{aligned}
    \| \hat{X}_{\epsilon_1,\epsilon_2}(\omega, \epsilon_1^{-1})-\exp(-i\epsilon_2\omega g(2\omega)\sigma_{y}) \|
    \leq& \| \hat{X}_{\epsilon_1,\epsilon_2}(\omega, \epsilon_1^{-1}) - \exp\bigl(-iI(\omega, \epsilon_1^{-1})\bigr) \| \\
    &+ \| \exp\bigl(-iI(\omega, \epsilon_1^{-1})\bigr) - \exp\bigl(-i\epsilon_2 \omega 
 g(2\omega)\sigma_{y}\bigr) \|.
\end{aligned}
\]
Following \eqref{eq:estimation_1} and \eqref{eq:estimation_2}, we conclude the proof by taking $C_n=\max(C''_{n},M')$.
\end{proof}

The next lemma connects the auxiliary system in \eqref{eq:auxiliary-system} to the original two-level driftless system~\eqref{system:main} on $\mathrm{SU}(2)$.

\begin{lemma}
\label{lem:one_step}

For $\epsilon_1,\epsilon_2>0$, take $T=1/\epsilon_1$ and consider the following controls
\begin{align}
    u\equiv \nu,\quad v(t)=\frac{\epsilon_2}{\sqrt{2\pi}}\hat{g}(t-T)
\end{align}
defined on the time interval $[0,2T]$, with $\nu\in\{-1,1\}$ and $\hat{g}(\cdot)$ defined in \eqref{eq:def_hat_g}. Applying this control to the system in \eqref{system:main}, we obtain that for all $\omega\in[v_0,v_1]$,
\begin{equation*}
    X(\omega,2T)=\exp(-i\nu\omega T\sigma_{x})\hat{X}_{\epsilon_1,\epsilon_2}(\nu\omega,T)\exp(-i\nu\omega T \sigma_{x}),
\end{equation*}
where $\hat{X}_{\epsilon_1,\epsilon_2}(\nu\omega,T)$ is the solution to the auxiliary system in \eqref{eq:auxiliary-system} at $t=T$.
\end{lemma}

\begin{proof}
    Fix $\epsilon_1,\epsilon_2>0$, let $T=1/\epsilon_1$, and consider the following change of variables:
    \begin{equation}
    \label{eq:change_of_variables_2}
        \bar{X}(\omega,t)=\exp(i\nu\omega t\sigma_{x})X(\omega,t+T).
    \end{equation}
    Under this transformation, the dynamics of $\bar{X}(\omega,\cdot)$ on the interval $[-T,T]$ satisfy
    \[
    \begin{aligned}
        i\dot{\bar{X}}\omega,t) = \hat{H}_{\epsilon_1,\epsilon_2}(\nu\omega,t)\bar{X}(\omega,t),\quad
        \bar{X}(\omega,-T) = \exp(-i\nu\omega T\sigma_x),
    \end{aligned}
    \]
    where $\hat{H}_{\epsilon_1,\epsilon_2}(\omega,t)$ denotes the time-dependent Hamiltonian defined in the auxiliary system~\eqref{eq:auxiliary-system}. It follows that
    \[
        \bar{X}(\omega,T)=\hat{X}_{\epsilon_1,\epsilon_2}(\nu\omega,T)\exp(-i\nu\omega T\sigma_{x}).
    \]
    We can then conclude the proof by applying the change of variables in \eqref{eq:change_of_variables_2} again.
\end{proof}

With the preceding lemmas in place, we are now ready to prove Theorem~\ref{theorem:main}.
\begin{proof}[{Proof of Theorem \ref{theorem:main}}]

Fix $\epsilon_1>0$, and $N, n\in\mathbb{N}^{*}$. Let us take $\epsilon_2=1/(2N)$ and apply Lemma~\ref{lem:one_step}. We then obtain that for all $\omega\in[v_0,v_1]$, the solution $X(\omega,\cdot)$ of system~\eqref{system:main} at $t=4(N+2)/\epsilon_1$ satisfies:
\[
    X\Bigl(\omega,\frac{4(N+2)}{\epsilon_1}\Bigr) = \Bigl(\hat{X}_{\epsilon_1,\epsilon_2}(-\omega, \epsilon_1^{-1})\hat{X}_{\epsilon_1,\epsilon_2}(\omega, \epsilon_1^{-1})\Bigr)^{N},
\]
where $\hat{X}_{\epsilon_1,\epsilon_2}(\omega,\cdot)$ is the solutions to \eqref{eq:auxiliary-system} with $\epsilon_2=1/(2N)$. By Lemma~\ref{lemma:one_step}, it follows that
\begin{equation*}
    \hat{X}_{\epsilon_1,\epsilon_2}(-\omega,\epsilon_1^{-1})\hat{X}_{\epsilon_1,\epsilon_2}(\omega,\epsilon_1^{-1})
    =\exp\Bigl(\frac{i\omega g(-2\omega)\sigma_{y}}{N}\Bigr) + \mathcal{O}(\epsilon_1^{n}N^{-1}+N^{-2}),
\end{equation*}
where $\mathcal{O}(\cdot)$ is uniform for all $\omega\in[v_0,v_1]$. Then by \eqref{eq:def_g}, $\omega g(2\omega)=-\omega g(-2\omega)=f(\omega)$ for all $\omega>0$, we have
\begin{equation*}
    \hat{X}_{\epsilon_1,\epsilon_2}(-\omega,\epsilon_1^{-1})\hat{X}_{\epsilon_1,\epsilon_2}(\omega,\epsilon_1^{-1}) \\
    =\exp\Bigl(-\frac{if(\omega)\sigma_{y}}{N}\Bigr)+\mathcal{O}(\epsilon_1^n N^{-1}+N^{-2}).
\end{equation*}

Finally, by taking the $N$-th power, we conclude that for all $\omega\in[v_0,v_1]$,
\[
\begin{aligned}
    X \Bigl(\omega, \frac{4N+2}{\epsilon_1}\Bigr)&= \bigl(\hat{X}_{\epsilon_1,\epsilon_2}(-\omega,\epsilon_1^{-1})\hat{X}_{\epsilon_1,\epsilon_2}(\omega,\epsilon_1^{-1})\bigr)^{N} \\
    &=\exp(-if(\omega)\sigma_{y})+\mathcal{O}(\epsilon_1^n+N^{-1}),
\end{aligned}
\]
where $\mathcal{O}(\cdot)$ is uniform for all possible $\omega$.
\end{proof}

\section{NUMERICAL RESULTS}
\label{sect:numerical}

In this section, we present numerical simulations to illustrate the effectiveness of the proposed control design strategy. The goal is to demonstrate that the constructed control inputs are capable of steering the entire quantum ensemble to the target distribution with high accuracy. These results serve to validate the theoretical guarantees established in the preceding section, and provide additional insight into the implementation of the method.

In this first example, we consider the ensemble control of system~\eqref{system:main} with $\omega\in[0.5,1]$. To define a smooth target distribution for the control design, we first introduce the function
\[
    p:\mathbb{R}\ni x\mapsto\left\{\begin{matrix}
        \exp\big(-\frac{1}{x}\big) & \text{if}\quad x>0,\\
        0  &\text{if}\quad x\leq 0
,    \end{matrix}\right.
\]
which is smooth on $\mathbb{R}$. Using this, we construct a smooth transition function
\[
    q:\mathbb{R}\ni x\mapsto \frac{p(x)}{p(x)+p(1-x)}
\]
satisfying $q(x)=0$ for $x\leq 0$ and $q(x)=1$ for $x\geq 1$. Given parameters $a<b<c<d$, this allows us to define the smooth bump function
\begin{equation}
\label{eq:bump}
    \Phi:\mathbb{R}\ni x\mapsto q\left(\frac{x-b}{a-b}\right)q\left(\frac{d-x}{d-c}\right),
\end{equation}
which satisfies $\Phi(x)=1$ for $x\in[b,c]$ and $\Phi(x)=0$ for $x\not\in [a,d]$. For this experiment, we choose $a=0.4,b=0.5,c=1$ and $d=1.1$, as illustrated in Figure~\ref{fig:bump}.

\begin{figure}[htbp]
    \includegraphics[width=0.45\linewidth]{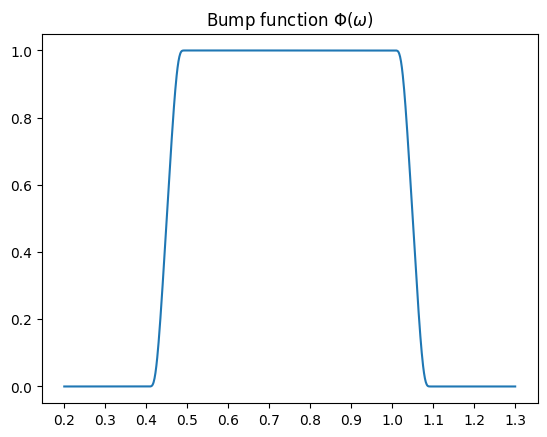}
    \centering
    \caption{Bump function in equation~\eqref{eq:bump} define with $a=0.4,b=0.5,c=1$ and $d=1.1$.}
    \label{fig:bump}
\end{figure}

The objective of this example is to steer all trajectories $X(\omega,\cdot)$ with $\omega\in[0.5,1]$ from the initial state $X(\omega,0)=\mathbb{I}$ to the common final state $\exp(-i\pi\sigma_{y}/2)$. To do so, instead of applying Theorem~\ref{theorem:main} directly to the discontinuous indicator function $\tfrac{\pi}{2}\cdot\mathbbm{1}_{[0.5,1]}$, we define a smooth approximation using the bump function: $f(\omega)=\tfrac{\pi}{2}\Phi(\omega)$, so that $f(\omega)=\tfrac{\pi}{2}$ for all $\omega\in[0.5,1]$ and $\operatorname{supp} f=[v_0,v_1]=[0.4,1.1]$. The corresponding smooth even functions $g$ and $\hat{g}$ are defined as in \eqref{eq:def_g} and \eqref{eq:def_hat_g}, yielding
\[
\begin{aligned}
    g(\omega) &= \frac{\pi}{\omega}\Phi\Bigl(\frac{\omega}{2}\Bigr), \\
    \hat{g}(t) &= \sqrt{2\pi}\int_{2v_0}^{2v_1}\Phi\Bigl(\frac{\omega}{2}\Bigr)\frac{\exp(-i\omega t)}{\omega}\text{d}\omega.
\end{aligned}
\]
Once the parameters $\epsilon_1>0$ and $N\in\mathbb{N}^{*}$ are fixed, the control inputs $u(\cdot)$ and $v(\cdot)$ are constructed according to Theorem~\ref{theorem:main}. To better visualize convergence, we recast the state of the system in a time-dependent frame defined by:
\[
    \hat{X}(\omega,t)=\exp\Bigl[i\Bigl(\int_{0}^{t}u(\tau)\text{d}\tau\Bigr)\omega\sigma_{x}\Bigr]X(\omega,t).
\]
We compare this with a reference trajectory
\[
    X\textsuperscript{ref}(\omega,t)=\exp\bigl(-ir(t)f(\omega)\sigma_{y}\bigr),
\]
where the piecewise constant function $r: [0,(4N+2)/\epsilon_1]\to [0,1]$ is defined as
\[
    r(t)=\frac{m}{2N}\quad\text{if}\quad t \in \Bigl(\frac{2m}{\epsilon_1},\frac{2m+1}{\epsilon_1}\Bigr].
\]
for a given $m\in\{0,\dots,2N\}$. Note that, by construction, $\int_{0}^{(4N+2)/\epsilon_1}v(\tau)\text{d}\tau=0$, which implies $\hat{X}(\omega, \tfrac{4N+2}{\epsilon_1}) = X(\omega, \tfrac{4N+2}{\epsilon_1})$. In the simulation, us choose $\epsilon_1=0.05$ and $N=10$. To evaluate performance, we consider
\[
\begin{aligned}
    P(\omega,t)&=|\langle \textbf{e}_2,\hat{X}(\omega,t)\textbf{e}_1\rangle|^{2},\quad\\P\textsuperscript{ref}(\omega,t)&=|\langle \textbf{e}_2,X\textsuperscript{ref}(\omega,t)\textbf{e}_1\rangle|^2,
\end{aligned}
\]
which represent the population on the second state $\textbf{e}_2$ of $\hat{X}(\omega,t)\textbf{e}_1$ and $X\textsuperscript{ref}(\omega,t)\textbf{e}_2$, respectively. These quantities are plotted for $\omega\in\{0.5,0.7,0.9\}$ in Figure~\ref{fig:example_1_1}. The final convergence of $\hat{X}(\omega,t)$ toward $\exp(-if(\omega)\sigma_{x})=\exp(-\pi\sigma_{y}/2)$ is illustrated in Figure~\ref{fig:example_1_2}.

\begin{figure}[thbp]
\label{fig:example_1}
    \centering
    \begin{subfigure}{0.4\textwidth}
\includegraphics[width=\textwidth]{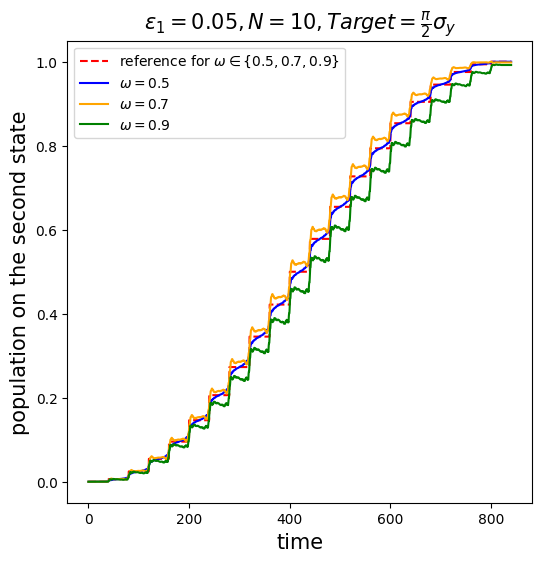}
    \caption{$P(\omega,t)$ and $P\textsuperscript{ref}(\omega,t)$ }
\label{fig:example_1_1}
    \end{subfigure}
    \hfill
    \begin{subfigure}{0.4\textwidth}
    \includegraphics[width=\textwidth]{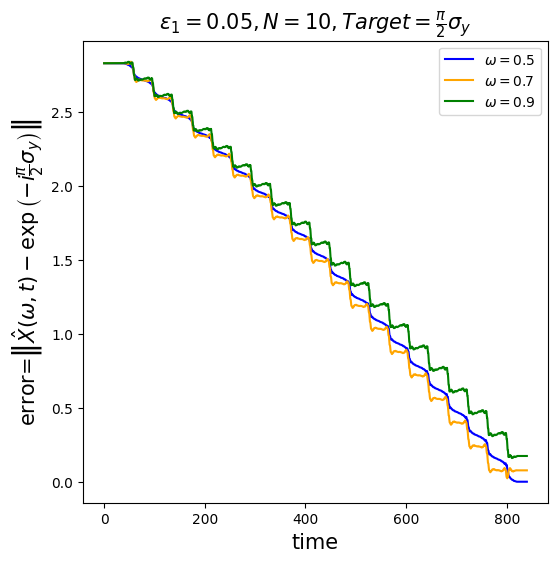}
    \caption{Convergence to target states}
\label{fig:example_1_2}
\end{subfigure}
\caption{Simulations for $\omega\in\{0.5,0.7,0.9\}$, $\epsilon_1=0.05,N=10$ and $f(\omega)=\frac{\pi}{2}\Phi(\omega)$: The three states are steered simultaneously to the parameterized final states $\exp\left(-i\frac{\pi}{2}\sigma_{y}\right)$with the same control $(u(\cdot),v(\cdot))$. Notice that in Figure~\ref{fig:example_1_1}, $P\textsuperscript{ref}(\omega,t)$ is the same for all $\omega$.}
\end{figure}

In the second numerical example, the objective is to steer the ensemble $X(\omega,t)$, for $\omega \in [0.5, 1]$, from the common initial state $X(\omega,0)=\mathbb{I}$ to the parameterized final states $\exp(-i\tfrac{\pi}{6\omega}\sigma_{x})$. To this end, we apply Theorem~\ref{theorem:main} with the function $f(\omega)=\tfrac{\pi}{6\omega}\Phi(\omega)$. Let $\tilde{u}(\cdot)$ and $\tilde{v}(\cdot)$ denote the control inputs generated by this choice of $f(\omega)$ according to the construction in Theorem~\ref{theorem:main}.

By applying the pair $(\tilde{u}(\cdot),\tilde{v}(\cdot))$, the system is driven to the final state $\exp(-i\frac{\pi}{6\omega}\sigma_{y})$, rather than the desired $\exp(-i\frac{\pi}{6\omega}\sigma_{x})$, due to the orientation of the control inputs. To correct this, we interchange the roles of the controls by setting $u(\cdot)=\tilde{{v}}(\cdot)$ and $v(\cdot)=\tilde{u}(\cdot)$. With this modification, the system is steered to the correct target states $\exp(-i\frac{\pi}{6\omega}\sigma_{x})$ for all $\omega \in [0.5, 1]$.

Similar to the convergence analysis in the previous example, we recast the state of the system in a time-dependent frame:
\[
    \hat{X}(\omega,t)=\exp\Bigl[i\Bigl(\int_{0}^{t}v(\tau)\text{d}\tau\Bigr)\omega\sigma_{y}\Bigr]X(\omega,t).
\]
The populations $P(\omega,t)$ and $P\textsuperscript{ref}(\omega,t)$ are defined analogously to the first example. For this simulation, we choose $\epsilon_1=0.025$ and $N=5$. The functions $P(\omega,t)$ and $P\textsuperscript{ref}(\omega,t)$ are compared in Figure~\ref{fig:example_2_1} for $\omega\in\{0.5,0.7,0.9\}$, and the convergence toward $\exp\left(-i\frac{\pi}{6}\sigma_x\right)$ is illustrated in Figure~\ref{fig:example_2_2}.

\begin{figure}[thp]
\label{fig:example_2}
    \centering
    \begin{subfigure}{0.4\textwidth}
    \includegraphics[width=\textwidth]{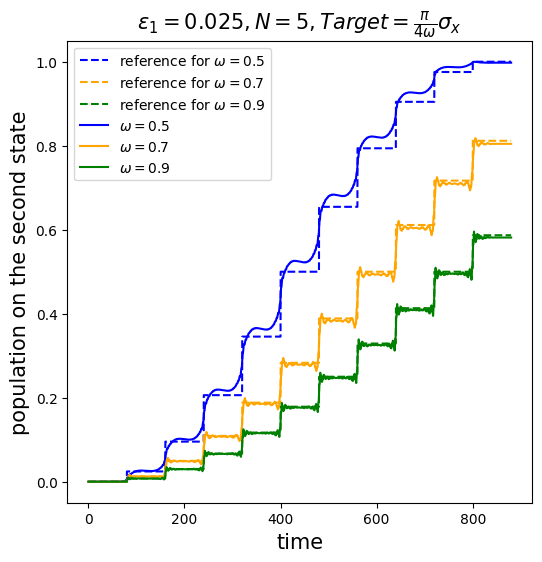}
    \caption{$P(\omega,t)$ and $P\textsuperscript{ref}(\omega,t)$ }
\label{fig:example_2_1}
    \end{subfigure}
    \hfill
    \begin{subfigure}{0.4\textwidth}
    \includegraphics[width=\textwidth]{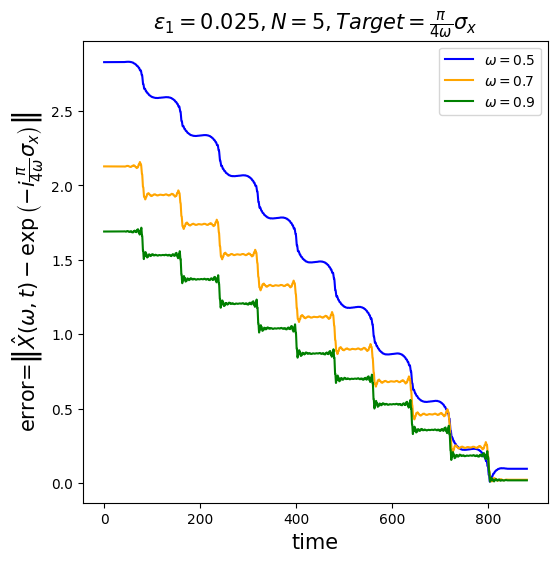}
    \caption{Convergence to target states}
\label{fig:example_2_2}
\end{subfigure}
\caption{Simulations for $\omega\in\{0.5,0.7,0.9\}$, $\epsilon_1=0.025,N=5$ and $f(\omega)=\frac{\pi}{4\omega}\Phi(\omega)$: The three states are steered simultaneously to the parameterized final states $\exp\left(-i\frac{\pi}{4\omega}\sigma_{x}\right)$with the same control $(u(\cdot),v(\cdot))$.}
\end{figure}

\section{CONCLUSION}

In this work, an explicit control design strategy was proposed for the ensemble control of a one-parameter family of driftless two-level quantum systems. Using Fourier transform techniques, it was shown that the proposed strategy can approximate any continuous ensemble on $\mathrm{SU}(2)$. The effectiveness of the method was illustrated through two numerical examples, which show that the ensemble system can be steered to closely approximate the desired target distribution within a small error. While previous studies have primarily focused on the theoretical aspects of ensemble controllability, relatively little attention has been given to the explicit construction of control signals. The present work contributes toward closing this gap by providing a constructive and implementable strategy. Future work will focus on ensemble control strategies for quantum systems with drift, $n$-level systems, and systems dependent on multiple parameters. Additional research will also explore quantum speed limits and optimal control design problems in the context of ensemble control.

% \section*{ACKNOWLEDGMENT}

% RL's work has been partly supported by the ANR-DFG project ``CoRoMo'' ANR-22-CE92-0077-01. This project has received financial support from the
% CNRS through the MITI interdisciplinary programs. GC's work is supported in part by the Fundamental Research Funds for the Central Universities and by NSFC Grant No.~12401588.

\bibliography{reference}

\end{document}